\documentclass[11pt]{amsart}
\usepackage[margin=1.25in]{geometry}
\usepackage{amsfonts,amsmath,amssymb,amsthm,graphics,subfig,enumerate,
  color, url, mathcomp,verbatim}
\usepackage{enumitem}

\DeclareMathOperator{\GL}{GL}
\DeclareMathOperator{\GSp}{GSp}
  
\DeclareMathOperator{\aut}{Aut}

\DeclareMathOperator{\re}{Re}
\DeclareMathOperator{\crk}{corank}
\DeclareMathOperator{\res}{Res}
\DeclareMathOperator{\coker}{cok}
\DeclareMathOperator{\rank}{rank}

\newcommand{\Prob}{\mathrm{Prob}}
\DeclareMathOperator{\HNF}{HNF}

\let\subset\subseteq

\newcommand{\Q}{\mathbb{Q}}
\newcommand{\Z}{\mathbb{Z}}
\newcommand{\N}{{\mathbb {N}}}
\newcommand{\oo}{{\mathcal {O}}}
\newcommand{\HH}{{\mathcal {H}}}

\newcommand{\qbinom}[2]{{#1 \brack #2}_q}
\newcommand{\pbinomreg}[2]{{#1 \brack #2}_p}

\newcommand{\pbinom}[2]{{#1 \brack #2}_{p^{-1}}}
\newcommand{\qfac}[1]{\left[#1\right]_q!}

\newcommand{\nn}{\nonumber}

\newtheorem{lemma}{Lemma}[section]
\newtheorem{prop}[lemma]{Proposition}
\newtheorem{theorem}[lemma]{Theorem}
\newtheorem{corollary}[lemma]{Corollary}

\theoremstyle{definition}

\newtheorem{definition}[lemma]{Definition}

\theoremstyle{remark}
\newtheorem*{proof*}{Proof}
\numberwithin{equation}{section}

\newenvironment{proofof}[1]{\noindent{\emph{Proof of {#1}.}}}{\qed\vspace{3ex}}

\title{The cotype zeta function of $\Z^d$}
\author{Gautam Chinta}
\address{Department of Mathematics, The City College of New York, New York, NY 10031}
\email{gchinta@ccny.cuny.edu}

\author{Nathan Kaplan}
\address{Department of Mathematics, University of California, Irvine, CA 92697-3875} 
\email{nckaplan@math.uci.edu}

\author{Shaked Koplewitz}
\address{Department of Mathematics, Yale University, New Haven, CT 06511} 
\email{shaked.koplewitz@yale.edu}

\begin{document}

\begin{abstract}
We give an asymptotic formula for the number of sublattices $\Lambda
\subseteq \Z^d$ of index at most $X$ for which $\Z^d/\Lambda$ has rank
at most $m$, answering a question of Nguyen and Shparlinski.  We
compare this result to work of Stanley and Wang on Smith normal
forms of random integral matrices and discuss connections to the
Cohen-Lenstra heuristics.  Our arguments are based on Petrogradsky's
formulas for the cotype zeta function of $\Z^d$, a multivariable
generalization of the subgroup growth zeta function of $\Z^d$.
\end{abstract}

\maketitle

\section{Introduction}

A fundamental problem in the field of subgroup growth is understanding
the number of subgroups of finite index $n$ in a fixed group $G.$  In
many cases, analytic properties of the {\em subgroup growth zeta
  function} $\zeta_G(s)$ provide useful information.  This is the
Dirichlet series
\begin{equation}
  \label{def:sgzf}
  \zeta_G(s)=\sum_{H\subset G} \frac{1}{[G:H]^s}
\end{equation}
where $H$ ranges over all finite index subgroups of $G$.  If the
number of subgroups in $G$ of index $n$ grows at most polynomially,
then the Dirichlet series defining $\zeta_G(s)$ converges absolutely
for $\re(s)$ sufficiently large.  An analytic continuation of the
series and knowledge of the locations and orders of its poles would provide information
on asymptotics for the number of subgroups of index less than $X$ as
$X\to\infty.$

One of the most basic examples 
is the subgroup growth zeta function of the integer lattice $\Z^d$
which turns out to have a simple expression as a product of Riemann
zeta functions:
\begin{equation}
  \label{eq:sgzf_Zd}
  \zeta_{\Z^d}(s)=\zeta(s)\zeta(s-1)\cdots \zeta(s-(d-1)).
\end{equation}
See the book of Lubotzky and Segal for five proofs
of this fact \cite{lubotzky_segal}. 
Since $\zeta(s)$ has a simple pole at $s=1$, standard Tauberian
techniques immediately give the asymptotic
\begin{equation}
  \label{eq:Zd_asymptotic}
  N_d(X):=\#\{\mbox{sublattices of } \Z^d
\mbox{ of index} <X\} 
= \frac{\zeta(d)\zeta(d-1)\cdots \zeta(2)}{d}X^d+O(X^{d-1}\log(X))
\end{equation}
as $X\to\infty$.  

\subsection{The proportion of lattices with given corank} A number of more refined questions about the distribution of
sublattices of $\Z^d$ can be asked.  Motivated by work of
Nguyen and Shparlinski \cite{ns}, we investigate the distribution of
sublattices of $\Z^d$ whose {\em cotype} has a certain form.  The
cotype of a sublattice $\Lambda\subset \Z^d$ is defined as follows.
By elementary divisor theory, there is a unique $d$-tuple of positive integers
$(\alpha_1,\ldots,\alpha_d)=(\alpha_1(\Lambda), \ldots,
\alpha_d(\Lambda))$ such that the finite abelian group $\Z^d/\Lambda$
is isomorphic to the sum of cyclic groups
\begin{equation}
  \label{eq:invt_factors}
  (\Z/\alpha_1\Z)\oplus
  (\Z/\alpha_2\Z)\oplus\cdots\oplus (\Z/\alpha_d\Z)
\end{equation}
where $\alpha_{i+1} \mid \alpha_i$ for $1\leq i\leq d-1.$ We call the $d$-tuple
$\alpha(\Lambda):=(\alpha_1(\Lambda), \ldots, \alpha_d(\Lambda))$ the
{\em cotype} of $\Lambda.$ The largest index $i$ for which
$\alpha_i(\Lambda)\neq 1$ is called the {\em rank} of $\Z^d/\Lambda$ and the {\em corank} of $\Lambda.$ By
convention, $\Z^d$ has corank 0.  A sublattice $\Lambda$ of corank 0
or 1 is called {\em cocyclic}, i.e.,  $\Z^d/\Lambda$ is cyclic, or
equivalently, $\Lambda$ has cotype $([\Z^d:\Lambda], 1, \ldots,
1)$.

Nguyen and Shparlinski study the distribution of
cocyclic sublattices of $\Z^d$ and pose several related
questions. Let $N_{d}^{(m)}(X)$ be the number of sublattices $\Lambda$
of $\Z^d$ of index less than $X$ such that $\Lambda$ has corank at most $m$.  In particular, $N_{d}^{(1)}(X)$ is
the number of cocyclic sublattices of $\Z^d$ of index less than $X$.  Throughout this paper we use $\prod_p$ to denote a product over all primes. Rediscovering a result of Petrogradsky \cite{petro} by more elementary
means, they show that
\begin{equation}
  \label{eq:ns_cocyclic}
 N_{d}^{(1)}(X)\sim \frac{\theta_d}{d}X^d, \mbox{ where } 
\theta_d=\prod_p\left( 1+\frac{p^{d-1}-1}{p^{d+1}-p^d}\right)
\end{equation}
as $X\to\infty$.  Comparing this to the asymptotic
(\ref{eq:Zd_asymptotic}) for all sublattices, Nguyen-Shparlinski and
Petrogradsky both observe that the probability that a ``random''
sublattice of $\Z^d$ is cocyclic is about $85\%$ for $d$ large.

Nguyen and Shparlinski conclude their paper by stating that it would
be of interest to obtain similar asymptotic formulas for
$N_d^{(m)}(X)$ for $m>1$ and to show that the sublattices of corank
$m$ form a negligible proportion of all sublattices of $\Z^d$ when $m$
is sufficiently large.

In this paper we show the following theorem.
\begin{theorem}\label{thm:ndmx}
  Let $1\leq m\leq d.$ As $X\to\infty$,
\begin{equation}
  \label{eq:ndmx}
  N_d^{(m)}(X)\sim\frac {X^d}{d}\cdot \prod_p \left((1-p^{-1})\sum_{i=0}^m 
\pbinom{d}{i}\frac{p^{-i^2}}{\prod_{j=1}^{i}{(1-p^{-j})}}\right).
\end{equation}
\end{theorem}
\noindent We recall in Section \ref{sec:petro} the definition of the $q$-binomial coefficient $\pbinom{d}{i}$.

Dividing by the number of all sublattices of index less than $X$ as
given in (\ref{eq:Zd_asymptotic}) gives the proportion of
sublattices with corank at most $m$.  

\begin{corollary} \label{cor:corank_m_density}
  As $X\to\infty$, 
\begin{equation}
  \label{eq:corank_m_density}
  \frac{N_d^{(m)}(X)}{N_d(X)}\sim
\prod_p \left(
\prod_{j=1}^{d}(1-p^{-j})\
\sum_{i=0}^m 
\pbinom{d}{i}\frac{p^{-i^2}}{\prod_{j=1}^{i}{(1-p^{-j})}}\right).
\end{equation}
\end{corollary}
\noindent 
For example, the proportion of sublattices of $\Z^d$ of corank at most 2
converges to approximately 99.4\% as $d\to\infty$, and the proportion of sublattices of $\Z^d$ of corank at most $3$ converges to approximately 99.995\%.  Therefore, while sublattices of any fixed corank have positive
density among all sublattices of $\Z^d$, they become sparser as the
corank grows.  This confirms an expectation of Nguyen-Shparlinski.

Also of interest in our work is the method of proof.  Nguyen and
Shparlinski prove their results by counting solutions of linear
congruence equations. Our proofs extend
Petrogradsky's methods and make systematic use of the {\em cotype zeta
  function} of $\Z^d$, which he introduced in \cite{petro}.  This is a
multivariate generalization of the subgroup growth zeta function
$\zeta_{\Z^d}(s)$ from (\ref{eq:sgzf_Zd}).  Petrogradsky computes it
explicitly in terms of permutation descent polynomials.

\subsection{Coranks of lattices and cokernels of matrices in Hermite normal form}  Throughout this paper, for a ring $R$ we let $M_d(R)$ denote the $d\times d$ matrices with entries in $R$. For a finite abelian group $G$, we write $(G)_p$ for its Sylow $p$-subgroup.  Consider the distribution on finite abelian $p$-groups of rank at most $d$ that chooses a group $G$ of rank $r$ where $r\le d$ with probability
\begin{equation}\label{rankd_prob}
P^p_d(G) = |\aut(G)|^{-1} \bigg(\prod_{j=1}^d (1-p^{-j})\bigg)  \bigg(\prod_{j=d-r+1}^d (1-p^{-j})\bigg).
\end{equation}
It follows from a result of Cohen and Lenstra \cite[Theorem 6.1]{cohen_lenstra} that the right-hand side of \eqref{eq:corank_m_density} is equal to the product over all primes $p$ of the probability that a group chosen from $P^p_d$ has rank at most $m$.

Motivated by famous conjectures of Cohen and Lenstra on distributions of Sylow $p$-subgroups of class groups of number fields \cite{cohen_lenstra}, Friedman and Washington prove that the distribution of cokernels of $d\times d$ random 
matrices with entries in the $p$-adic integers $\Z_p$, drawn from Haar measure on the space of all such matrices, is the distribution of \eqref{rankd_prob} \cite[Proposition 1]{fw}. Stanley and Wang show that this distribution arises in the study of Smith normal forms of random $d \times d$ integer matrices with entries chosen uniformly from $[-k,k]$, as $k \to \infty$.  The Smith normal form of an integer matrix carries the same information as its cokernel.  As $k \to \infty$, each entry is uniformly
distributed modulo $p^r$ for each prime power, so
this distribution of cokernels matches the one studied by Friedman and
Washington, and therefore is equal to the one defined by
\eqref{rankd_prob}.  Going from a result for a single prime to a result involving infinitely many primes is often challenging.  Stanley and Wang prove that the probability that the cokernel of a random integer matrix chosen from the model described above has rank at most $m$ is given by the right-hand side of \eqref{eq:corank_m_density}  \cite[Theorem 4.13]{sw}.  The proof uses nontrivial results from number theory of Ekedahl and Poonen on greatest common divisors of outputs of multivariable polynomials \cite{ek, poonen}.

We now interpret of Corollary \ref{cor:corank_m_density} in terms of cokernels of special classes of random integer matrices.  A nonsingular $M \in M_{d}(\Z)$ with entries $a_{ij}$ is in \emph{Hermite normal form} if:
\begin{enumerate}
\item $M$ is upper triangular, and
\item $0 \le a_{ij} < a_{jj}$ for $1 \le i < j \le d$.
\end{enumerate}
We recall a basic fact about lattices and matrices in Hermite normal form.  
\begin{prop}\label{hnf_prop} 
Every sublattice $\Lambda \subseteq \Z^d$ is the row span of a unique matrix $H(\Lambda)$ in Hermite normal form.  Moreover, $\Z^d/\Lambda \cong \coker(H(\Lambda))$, which implies $[\Z^d\colon \Lambda] = \det(H(\Lambda))$.  

This gives a bijection between the set of sublattices of $\Z^d$ of index less than $X$ and nonsingular $d \times d$ matrices in Hermite normal form with determinant less than $X$.
\end{prop}
Let $\HH_d(\Z) \subset M_d(\Z)$ denote the subset of nonsingular matrices in Hermite normal form and $\HH_d(X) \subset \HH_d(\Z)$ denote the subset of these matrices with determinant less than $X$.  By Proposition \ref{hnf_prop}, Corollary \ref{cor:corank_m_density} is equivalent to the following statement.
\begin{corollary}\label{HNF_prank}
We have
\[
\lim_{X\rightarrow \infty} \frac{\#\{M \in \HH_d(X) \colon \rank(\coker(M)) \le m \} }{\#\HH_d(X) } = \prod_p \left(
\prod_{j=1}^{d}(1-p^{-j})\
\sum_{i=0}^m 
\pbinom{d}{i}\frac{p^{-i^2}}{\prod_{j=1}^{i}{(1-p^{-j})}}\right).
\]
\end{corollary}

In Section \ref{sec:universality} we consider the distribution of Sylow $p$-subgroups of cokernels of matrices in Hermite normal form, giving an explanation for this result. 
\begin{theorem}\label{dist_ppart}
Let $G$ be a finite abelian $p$-group of rank $r \le d$.  Then 
\[
\lim_{X\to\infty} \frac{\#\left\{\Lambda \subseteq \Z^d \colon [\Z^d\colon \Lambda] < X \text{ and }  (\Z^d/\Lambda)_p \cong G\right\}}{\#\left\{\Lambda \subseteq \Z^d \colon [\Z^d\colon \Lambda] < X \right\}} = P^p_d(G).
\]
Equivalently,
\[
\lim_{X\to\infty} \frac{\#\left\{M \in \HH_d(X) \colon \coker(M)_p \cong G\right\}}{\#\HH_d(X)} = P^p_d(G).
\]
\end{theorem}
\noindent We note that this result does not directly imply Corollary \ref{HNF_prank} because of subtleties involved in going from a single prime to a product over all primes.

The main point of Theorem \ref{dist_ppart} is that Sylow $p$-subgroups of cokernels of matrices in Hermite normal form follow the same distribution as Sylow $p$-subgroups of cokernels of Haar random matrices in $M_d(\Z_p)$.  This fits in with universality results for cokernels of families of random integer and $p$-adic matrices due to Wood \cite{wood}.  However, Wood's results do not directly imply Theorem \ref{dist_ppart} because, in the language of \cite[Definition 1]{wood}, matrices in Hermite normal form are not $\epsilon$-balanced, as many entries are fixed to be $0$.

\subsection{Related work}

More general functions of the type considered in this paper have their
origin in Igusa's study of zeta functions of algebraic groups
\cite{I}.  This work has become an essential tool in the theory of
zeta functions of groups and rings and has been extended in various
directions.  See for example the paper of du Sautoy and Lubotzky
\cite{duSL} and the further references listed in Section
\ref{subsec:classical}.  In this context, Petrogradsky's local zeta
function and generalizations thereof arise naturally as multivariate
$p$-adic integrals.  This point of view is developed in Voll
\cite{Voll}, in both the context of counting subgroups of a finitely
generated torsion-free nilpotent group and in the study of their
representation zeta functions.  Further generalizations of Igusa's local zeta functions are introduced in Klopsch-Voll \cite{KlopschVoll} and Schein-Voll \cite{MR3338607}, leading to numerous applications in subgroup, subring and represention growth; see e.g. \cite{MR4138699,MR3846349,MR3856853,ScheinVollII,StasinskiVollB,StasinskiVoll1}.

In the theory of autorphic forms related functions appear in the
computation of Fourier coefficients of Eisenstein series on $GL_{2n}$ and 
the local singular series of an $n$-by-$n$ square matrix, as noted in
the papers of F. Sato \cite{sato} and Beineke-Bump \cite{bb}.  Sato
raises the interesting open question of finding a corresponding
relation between local singular series and the
enumeration of subgroups of an abelian group in the symplectic case.

\subsection*{Outline of the paper}

We review Petrogradsky's work in Section \ref{sec:petro}.  In Section
\ref{sec:proofs} we prove our main results on the distribution of the
corank.  In Section \ref{sec:universality} we prove Theorem \ref{dist_ppart}. 
The utility of the cotype zeta function in the
resolution of these corank problems suggests that it may be fruitful
to introduce multivariate Dirichlet series to address analogous
subgroup and subring growth problems in a broader context.  We
elaborate on this and present some further concluding remarks in
Section \ref{sec:conclusion}.

\subsection*{Acknowledgments}

It is a pleasure to thank Professors G. Bhowmik, M. du Sautoy,
J. Fulman, S. Payne, V. Petrogradsky, I. Shparlinski, and M. Wood for
their encouragement and suggestions.  We are also grateful to
Professor C. Voll for his very helpful comments on an earlier version of
this paper.  The first named author is supported by NSF grant DMS 1601289
  and PSC-CUNY grant 69658-00 47.
  The second named author received support from NSA Young Investigator Grant
  H98230-16-10305, an AMS-Simons Travel Grant, and NSF grant DMS 1802281.
    The third named author was supported by NSF grant DMS 1149054.

\section{The cotype zeta function of $\Z^d$}
\label{sec:petro}

We recall Petrogradsky's definition of the \emph{cotype zeta function} of $\Z^d$, which he calls the \emph{multiple zeta function} of $\Z^d$.  
\begin{definition}\cite{petro}\label{petro_zeta_def}
Let $d$ be a positive integer and let $a_{\alpha}(\Z^d)$ be the number of subgroups $\Lambda \subseteq \Z^d$ of cotype $\alpha$.  We define the \emph{cotype zeta function} of $\Z^d$:
\[
\zeta_{\Z^d}(s_1,\ldots, s_d) = \sum_{H \subseteq \Z^d} \alpha_1(H)^{-s_1}\cdots \alpha_d(H)^{-s_d} = \sum_{\alpha = (\alpha_1,\ldots, \alpha_d)} a_{\alpha}(\Z^d) \cdot \alpha_1^{-s_1}\cdots \alpha_d^{-s_d}.
\]
\end{definition}
\noindent Note that $\zeta_{\Z^d}(s,\ldots, s) = \zeta_{\Z^d}(s)$, so
this multivariable function generalizes the subgroup growth zeta
function of $\Z^d$.

The subgroup growth zeta function of $\Z^d$ has an Euler product, and Petrogradsky shows that this multivariable generalization has one as well.  
\begin{lemma}\cite[Lemma 1.1]{petro}
For each $d \ge 1$ we have
\[
\zeta_{\Z^d}(s_1,\ldots, s_d) = \prod_p \zeta_{\Z^d,p}(s_1,\ldots, s_d),
\]
where the local factor for each prime $p$ is defined as
\[
\zeta_{\Z^d,p}(s_1,\ldots, s_d) = 
\sum_{m=0}^\infty \sum_{H \subseteq \Z^d \atop 
\left[\Z^d:H\right] = p^m} \alpha_1(H)^{-s_1}\cdots \alpha_d(H)^{-s_d}.
\]
\end{lemma}

One of the main results of \cite{petro} is the computation of the local factors of the cotype zeta function of $\Z^d$ in terms of permutation descents and $q$-binomial coefficients.  We fix some notation and recall basic properties of these combinatorial objects following \cite[Section 3]{petro}:
\begin{eqnarray*}
[n]_q & = & \frac{1-q^n}{1-q} = 1 + q+ \cdots + q^{n-1}; \\
\qfac{n} & = & [n]_q [n-1]_q \cdots [2]_q; \\
\qbinom{n}{j} & = &  \frac{\qfac{n}}{\qfac{j} \qfac{n-j}};\\
\qbinom{m_1+m_2+\cdots + m_k}{m_1, m_2,\ldots, m_k} & = & 
\frac{\qfac{m_1+m_2+\cdots + m_k}}{\qfac{m_1} \qfac{m_2}\cdots \qfac{m_k}}\\
& = & \qbinom{m_1+ m_2 + \cdots + m_k}{m_1} \qbinom{m_2+\cdots + m_k}{m_2} \cdots \qbinom{m_{k-1} + m_k}{m_{k-1}}.
\end{eqnarray*}
Let $\N_d = \{1,2,\ldots, d\}$ and suppose $\lambda \subseteq \N_d$.  If $\lambda = \emptyset$ then we set $\qbinom{d}{\lambda} = 1$.  Otherwise, if $\lambda = \{\lambda_1,\ldots, \lambda_k\}$, where $d \ge \lambda_1 > \lambda_2 > \cdots > \lambda_k \ge 1$, let $|\lambda| = k$ and $m_i = \lambda_i - \lambda_{i+1}$ for $0\le i \le k$, where we set $\lambda_{k+1} = 0$ and $\lambda_0 = d$.  Note that $d = m_0 + m_1 + \cdots + m_k$.  We define the following polynomials in $q$:
\begin{eqnarray}
\qbinom{d}{\lambda} & = &  \qbinom{d}{m_0, m_1, \ldots, m_k},\ \ \ \lambda \subseteq \N_d; \\
w_{d,\lambda}(q) & = &  \sum_{\mu \subseteq \lambda} (-1)^{|\lambda| - |\mu|} \qbinom{d}{\mu},\ \ \ \lambda \subseteq \N_d.
\end{eqnarray}

\begin{theorem}\cite[Theorem 3.1]{petro}\label{petro_main_thm}
Consider the cotype zeta function of $\Z^d$.  For each $j$ satisfying $1\le j \le d$, we introduce the new variable $z_ j = s_1 + \cdots + s_j - j(d-j)$. Then
\begin{enumerate}
\item 
\begin{eqnarray}
\zeta_{\Z^d}(s_1,\ldots, s_d) & = & \zeta(z_1)\zeta(z_2) \cdots \zeta(z_d) \cdot f(s_1,\ldots, s_{d-1}),\ \ \text{where}\label{cotype_eqn} \\
f(s_1,\ldots, s_{d-1}) & = & \prod_p \left(1 + \sum_{\emptyset \neq \lambda \subseteq \N_{d-1}} w_{d,\lambda}\big(p^{-1}\big) \prod_{j \in \lambda} p^{-z_j} \right),
\end{eqnarray}
with the sum taken over all nonempty subsets $\lambda \subseteq \N_{d-1}$.

\item The region of absolute convergence of the cotype zeta function is $\re(z_j) > 1,\\ 
j = 1, \ldots, d$.
\item The region of absolute convergence of the product over primes is $\re(z_j) > 0, \\ 
j = 1, \ldots, d-1$; and (\ref{cotype_eqn}) is the analytic continuation to this region.

\item The local factors are rational functions in the variables $t_i = p^{-z_i},\ i = 1,\ldots, d$:
\begin{equation}\label{eq:zetap_local}
\zeta_{\Z^d,p}(s_1,\ldots, s_d) = \frac{\sum_{\lambda \subseteq\N_{d-1}} w_{d,\lambda}(p^{-1}) \prod_{j\in \lambda} t_j}{(1-t_1) (1-t_2) \cdots (1-t_d)}.
\end{equation}

\end{enumerate}
\end{theorem}

The polynomials $w_{d,\lambda}(q)$ that arise have been studied
extensively in the combinatorial literature.  The first part of
Theorem \ref{thm_perm} below is stated in \cite[Theorem 3.1
(2)]{petro}, while the other two parts are due to Stanley
\cite[Corollary 3.2]{stanley_jcta} and \cite[Section 2.2.5]{stanley_ec1}.

\begin{definition}
  Let $\pi \in S_d$ be a permutation.  We call $i \in \{1,2\ldots,
  d-1\}$ a \emph{descent} of $\pi$, provided that $\pi(i) > \pi(i+1)$.
  For $\pi \in S_d$, let $D'(\pi)$ denote its set of descents.  

  A pair $(i,j)$ is called an \emph{inversion} of $\pi$ if and only if
  $i < j$ and $\pi(i) > \pi(j)$.  Let $\operatorname{inv}(\pi)$ denote
  the number of inversions of $\pi$. 
\end{definition}
\noindent Note that $d$ cannot be a descent of a permutation in $S_d$.

\begin{theorem}\label{thm_perm}
Let $\lambda \subseteq \N_{d-1}$ be fixed.
\begin{enumerate}
\item There exists a number $N \ge |\lambda|$ such that $w_{d,\lambda}(q)$ is a polynomial in $q$ with nonnegative integer coefficients of the form 
\[
w_{d,_\lambda}(q) = q^N + \sum_{i > N} \tau_i q^i.
\]

\item We have that 
\[
w_{d,\lambda}(q) = \sum_{\pi \in S_d \atop D'(\pi) = \lambda} q^{\operatorname{inv}(\pi)}.
\]

\item We have that 
\[
w_{d,\lambda}(q) = \det\left( \qbinom{d-\lambda_{i+1}}{\lambda_j-\lambda_{i+1}}\ \mid\ 0 \le i,j \le k\right).
\]
\end{enumerate}

\end{theorem}

To conclude this section, we compare the results of Petrogradsky
described here to the work of du Sautoy and Lubotzky \cite{duSL}, which builds on earlier work of Igusa \cite{I}.
Theorem 5.9 of \cite{duSL}, specialized to $G=\GL_d$ and $\rho$ the
standard representation gives (\ref{eq:zetap_local}).  (The result of
\cite{duSL} is specialized to a single variable, but the multivariate
extension is obvious.)  Petrogradsky's proof uses a cotype-preserving bijective
correspondence between finite index subgroups $\Lambda$ of $\Z^d$ and subgroups
of the finite group $\Z^d/N\Z^d,$ where $N=\alpha_1(\Lambda).$  The
number of the latter can be expressed in terms of $q$-binomial
coefficients \cite{butler}.  On the other hand, du Sautoy and Lubotzky
interpret the $p$-part of the zeta function as a $p$-adic integral
over $\GL_d(\Z_p)$, which they compute using the Iwahori
decomposition.  This leads to a sum over the (affine) Weyl group
equivalent to (\ref{eq:zetap_local}). These ideas have been further developed to prove local functional equations for zeta functions of nilpotent groups and other Igusa-type zeta functions; see, e.g., \cite{KlopschVoll, Voll}.

\section{Density results for the corank}
\label{sec:proofs}

We begin by introducing the Dirichlet series counting
sublattices of $\Z^d$ of corank less than or equal to $m$.  This is
given by
\begin{equation}
  \label{def:corank_zeta}
  \zeta_{\Z^d}^{(m)}(s)=
\sum_{\substack{\Lambda\subset\Z^d \\ \crk(\Lambda)\leq m }}
\frac{1}{[\Z^d:\Lambda]^s}
\end{equation}
Recall that a sublattice of corank at most $m$
will have cotype $(\alpha_1, \alpha_2, \ldots, \alpha_d)$ with
$\alpha_{m+1}=\cdots=\alpha_d=1$. Therefore, in terms of
Petrogradsky's expression for the cotype zeta function given in Theorem \ref{petro_main_thm},  we have
\begin{align}
  \label{eq:zeta_rank_m}\nn
  \zeta_{\Z^d}^{(m)}(s)&=
\lim_{s_{m+1}\to\infty}\cdots\lim_{s_d\to\infty}
\zeta_{\Z^d}(s, \ldots ,s, s_{m+1}, \ldots, s_d) \\
&=\zeta(s-(d-1))\zeta(2s-2(d-2))
\cdots\zeta(ms-m(d-m))f_d^{(m)}(s),
\end{align}
where 
\[f_d^{(m)}(s)=\prod_p f_{d,p}^{(m)}(s)=
\prod_p 
\left(\sum_{\lambda\subset\{1,\ldots,m\}} 
\omega_{d,\lambda}(p^{-1}) \prod_{j\in\lambda}
p^{-js+j(d-j)}
\right).
\]
The analytic properties of $\zeta_{\Z^d}^{(m)}(s)$ will lead to our
desired density results.

\begin{prop}\label{prop:residue_corank_zeta}
  The corank at most $m$ zeta function $\zeta_{\Z^d}^{(m)}(s)$ has a simple
  pole at $s=d$ with residue
\begin{equation}
  \label{eq:2}
 \prod_p \left((1-p^{-1})\sum_{i=0}^m 
\qbinom{d}{i}\frac{p^{-i^2}}{\prod_{j=1}^{i}{(1-p^{-j})}}\right).
\end{equation}
\end{prop}

The simple pole comes from the simple pole of the Riemann zeta
function $\zeta(s-(d-1))$ at $s=d$.  The other zeta factors in
(\ref{eq:zeta_rank_m}) are holomorphic at $s=d$ and collectively
contribute a factor of $\prod_{2\leq j\leq m}\zeta(j^2)$ at $s=d$ to
the residue.  Thus
\begin{equation}
  \label{eq:9}
  \res_{s=d}\zeta_{\Z^d}^{(m)}(s) = \left(\prod_{2\leq j\leq m}\zeta(j^2)
\right) f_{d}^{(m)}(d).
\end{equation}
To complete the proof of Proposition \ref{prop:residue_corank_zeta}, it remains to evaluate
\begin{equation}
f_{d,p}^{(m)}(d)=\sum_{\lambda\subset\{1,\ldots,m\}}
\omega_{d,\lambda}(p^{-1}) \prod_{j\in\lambda} p^{-j^2}
\end{equation} 
and take the product over all primes $p$. Setting $q=p^{-1}$, we compute
  \begin{align}\nn
    \sum_{\lambda\subset\{1,\ldots,m\}} \omega_{d,\lambda}(q)
  \prod_{j\in\lambda} q^{j^2} &=
 \sum_{\lambda\subset\{1,\ldots,m\}}
\left(\sum_{\mu\subset\lambda} (-1)^{|\lambda|-|\mu|}\qbinom{d}{\mu}
\right)\prod_{j\in\lambda} q^{j^2}\\ \nn
& = \sum_{\mu\subset\{1,\ldots,m\}}\qbinom{d}{\mu}
\left(\sum_{\lambda\supset\mu} (-1)^{|\lambda|-|\mu|}
\prod_{j\in\lambda} q^{j^2}\right)\\ \label{eq:prop31id}
& = \sum_{\mu\subset\{1,\ldots,m\}}\qbinom{d}{\mu}
\prod_{j\in\mu} q^{j^2}
\prod_{j\not\in\mu} (1-q^{j^2}).
  \end{align}

In order to go further we need the intermediate result of Lemma
\ref{lemma:qid2} below.

\subsection{A $q$-multinomial identity}
The lemma below follows from the $q$-binomial theorem.
\begin{lemma}\label{lemma:qid}
Let $e, n$ be nonnegative integers.  We have
\begin{equation}
  \label{eq:lemma}
  \sum_{k=0}^n \qbinom{n}{k}q^{k^2+ek}  
\prod_{j= k+1+e}^{n+e}(1- q^{j})
= 1.
\end{equation}
\end{lemma}

This lemma will be used in the proof of Lemma \ref{lemma:qid2} below.
We note in passing that setting $e=0$ and letting $n\to\infty$ yields
the generating series for partitions in terms of the Durfee number
generating series. 

\begin{lemma}\label{lemma:qid2}
  Let $1\leq i\leq d$ be integers. We have
\begin{equation}
  \label{eq:1}
  \sum_{\mu\subset\{1,\cdots,i-1\}}
\qbinom{d}{\mu\cup\{i\}}\ \prod_{j\in \mu} q^{j^2}
\prod_{j\in\{1,\cdots,i-1\}\backslash \mu}(1- q^{j^2})
=
\qbinom{d}{i}\ \prod_{j=1}^{i}\frac{1-q^{j^2}}{1-q^j}.
\end{equation}
\end{lemma}

The argument below is similar to that given in Section 4.1 of
Stasinski-Voll \cite{StasinskiVoll-statistic}.  We give a proof for completeness.

\begin{proofof}{Lemma \ref{lemma:qid2}}
  We argue by induction on $i$.  The base case $i=1$ is immediate.
  Assume the identity is true for all $i_0$ satisfying $1\leq i_0<i$.  We remove the contribution of $\mu=\emptyset$ from the left-hand side of
(\ref{eq:1})  and write it as
\begin{align} \nn
   \sum_{\mu\subset\{1,\cdots,i-1\}}&
\qbinom{d}{\mu\cup\{i\}} 
\ \prod_{j\in \mu} q^{j^2}
\prod_{j\in \{1,\cdots,i-1\}\backslash\mu}(1- q^{j^2}) -
\qbinom{d}{i}\prod_{j=1}^{i-1}(1- q^{j^2})  \\ \nn
&= \sum_{i_0=1}^{i-1}\  \ \sum_{\mu\subset\{1,\cdots,i_0-1\}}
\qbinom{d}{\mu\cup\{i_0,i\}} \ 
\prod_{j\in \mu\cup\{i_0\}}q^{j^2}
\prod_{j\in\{1,\cdots,i-1\}\backslash \mu\cup\{i_0\}}(1- q^{j^2})\\ \nn
&= \sum_{i_0=1}^{i-1} \qbinom{d-i_0}{i-i_0} q^{i_0^2}  
\prod_{j\in i_0+1}^{i-1}(1- q^{j^2})\\ \label{eq:innersum}
 & \qquad\qquad \times\sum_{\mu\subset\{1,\cdots,i_0-1\}}
\qbinom{d}{\mu\cup\{i_0\}} \ 
\prod_{j\in \mu}q^{j^2}
\prod_{j\in\{1,\cdots,i_0-1\}\backslash \mu}(1- q^{j^2}),
\end{align}
where we have used the identity
\begin{equation*}
  \qbinom{d}{\mu\cup\{i_0,i\}}=\qbinom{d}{\mu\cup\{i_0\}}
\qbinom{d-i_0}{i-i_0}
\end{equation*} in the final step.
We continue by using the inductive hypothesis on the inner sum, i.e., the expression in
(\ref{eq:innersum}), and see that the left-hand side of (\ref{eq:1}) is equal to: 
\begin{align*}
\qbinom{d}{i}\prod_{j=1}^{i-1}(1- q^{j^2}) +
\sum_{i_0=1}^{i-1} & \qbinom{d-i_0}{i-i_0} q^{i_0^2}  
\prod_{j= i_0+1}^{i-1}(1- q^{j^2})
\qbinom{d}{i_0}\ \prod_{j=1}^{i_0}\frac{1-q^{j^2}}{1-q^j}\\
&= \prod_{j= 1}^{i-1}(1- q^{j^2})
\sum_{i_0=0}^{i-1} \qbinom{d}{i}\qbinom{i}{i_0}q^{i_0^2}  
\prod_{j= 1}^{i_0}(1- q^{j})^{-1}
\end{align*}
where we have used the subset-of-a-subset identity.
Comparing with the right-hand side of (\ref{eq:1}), we are reduced to proving
\begin{equation*}
  \label{eq:3}
\sum_{i_0=0}^{i-1} \qbinom{i}{i_0}q^{i_0^2}  
\prod_{j= 1}^{i_0}(1- q^{j})^{-1}
= \frac{1-q^{i^2}}{\prod_{j=1}^{i}(1-q^j)}
\end{equation*}
or equivalently,
\begin{equation*}
  \label{eq:4}
\sum_{i_0=0}^{i-1} \qbinom{i}{i_0}q^{i_0^2}  
\prod_{j= i_0+1}^{i}(1- q^{j})
= 1-q^{i^2}.
\end{equation*}
We can write this a little more nicely:
\begin{equation*}
  \label{eq:5}
\sum_{i_0=0}^i \qbinom{i}{i_0}q^{i_0^2}  
\prod_{j= i_0+1}^{i}(1- q^{j})
= 1.
\end{equation*}
This is the case $e=0$ of Lemma \ref{lemma:qid}.
\end{proofof}

\subsection{Conclusion of the proof of Proposition
  \ref{prop:residue_corank_zeta}}

We return to the evaluation of $f_{d,p}^{(m)}(d)$ using the expression of
(\ref{eq:prop31id}):
\begin{align*}
 f_{d,p}^{(m)}(d)&= \sum_{\mu\subset\{1,\ldots,m\}}\qbinom{d}{\mu}
\prod_{j\in\mu} q^{j^2}
\prod_{j\not\in\mu} (1-q^{j^2}).
\end{align*}
By Lemma \ref{lemma:qid2}, the above sum restricted to subsets with
largest element $i$ yields \small
\begin{align*}
 \left[ q^{i^2}\prod_{j=i+1}^m(1-q^{j^2})\right]
\sum_{\mu\subset\{1,\ldots,i-1\}}
\qbinom{d}{\mu\cup\{i\}}\ \prod_{j\in \mu} q^{j^2}
\prod_{j\in\{1,\cdots,i-1\}\backslash \mu}(1- q^{j^2})=
\qbinom{d}{i}\frac{q^{i^2}\prod_{j=1}^m(1-q^{j^2})}{\prod_{j=1}^{i}{1-q^j}}.
\end{align*}\normalsize
Noting that $i=0$ corresponds to the contribution of $\mu=\emptyset$,
we sum over all $i$ to obtain
\begin{equation*}
  \label{eq:8}
  f_{d,p}^{(m)}(d)= \left[\prod_{j=1}^m(1-q^{j^2})\right]
\sum_{i=0}^m 
\qbinom{d}{i}\frac{q^{i^2}}{\prod_{j=1}^{i}{(1-q^j)}}.
\end{equation*}
Now taking the product over $p$ cancels the zeta factors in
(\ref{eq:9}) and we are left with
\begin{equation*}
  \label{eq:7}
  \res_{s=d}\zeta_{\Z^d}^{(m)}(s)= \prod_p \left((1-p^{-1})\sum_{i=0}^m 
\pbinom{d}{i}\frac{p^{-i^2}}{\prod_{j=1}^{i}{(1-p^{-j})}}\right).
\end{equation*}
This concludes the proof of Proposition \ref{prop:residue_corank_zeta}.

\subsection{The density of sublattices of corank at most $m$}

Theorem \ref{thm:ndmx}, the asymptotic expression for the number of
sublattices with corank at most $m$, follows immediately from
Proposition \ref{prop:residue_corank_zeta} and the analytic
continuation statements from Theorem \ref{petro_main_thm}.

We note that the constant term in the expression
(\ref{eq:Zd_asymptotic}) is
\[
\frac{\zeta(d)\zeta(d-1)\cdots \zeta(2)}{d} = 
\frac{1}{d} \prod_p \prod_{j=0}^{d-2} \left(1-p^{-d+j}\right)^{-1}.
\]
Taking the quotient of this term with the constant term in Theorem \ref{thm:ndmx} completes the proof of Corollary \ref{cor:corank_m_density}.

\section{Sylow $p$-subgroups of cokernels of matrices in Hermite normal form}
\label{sec:universality}

The goal of this section is to prove the second of the equivalent statements in Theorem \ref{dist_ppart}.  Our strategy for determining the distribution of Sylow $p$-subgroups of cokernels of matrices in Hermite normal form is to relate this distribution to the distribution of cokernels of Haar random $p$-adic matrices.  Haar measure on the $p$-adic integers $\Z_p$ gives rise to Haar measure on $M_d(\Z_p)$, normalized so that the total volume is $1$.  More concretely, each matrix entry can be chosen independently with respect to Haar measure on $\Z_p$.  Throughout the rest of this section, we use $\Prob_{M \in M_d(\Z_p)}(\cdot)$ to denote the probability that a Haar random matrix $M \in M_d(\Z_p)$ has some property.  This is equal to the volume of the subset of $M_d(\Z_p)$ consisting of matrices with this property.  We give an example from the introduction.  Recall the distribution $P^p_d$ on finite abelian $p$-groups of rank at most $d$ defined in (\ref{rankd_prob}).
\begin{prop}\cite[Proposition 1]{fw}\label{FW_prop}
Let $G$ be a finite abelian $p$-group of rank at most $d$.  Let $M \in M_d(\Z_p)$ be a random matrix.  Then
\[
\Prob_{M \in M_d(\Z_p)}(\coker(M)\cong G) = P^p_d(G).
\]
\end{prop}

We use the following fact, which follows from Proposition \ref{FW_prop} and \cite[Corollary 3.8]{cohen_lenstra}. 
\begin{prop}\label{CL_cor}
Let $e \ge 0$ be an integer and $M \in M_d(\Z_p)$ be a random matrix.  Then
\[
\Prob_{M \in M_d(\Z_p)}(\left|\coker(M)\right| = p^e) = \frac{\prod_{j=1}^d (1-p^{-j})}{p^e} \pbinom{d+e-1}{e}.
\]
\end{prop}
\noindent Any $\alpha \in \Z_p$ can be written uniquely as $\alpha = p^e u$ where $e\in \Z_{\ge 0}$ and $u\in \Z_p$ is a unit.  In this case, we write $v_p(\alpha) = e$. Note that $\left|\coker(M)\right| = p^e$ if and only if $v_p(\det(M)) = e$.

We will use the following analogue of Proposition \ref{hnf_prop} for matrices with entries in $\Z_p$.
\begin{prop}\cite[Theorem 3.1.7]{caruso}\label{Zp_HNF}
Any $M \in M_d(\Z_p)$ can be written uniquely as a product $M = U H$ where $U \in \GL_d(\Z_p)$ and $H$ is an upper triangular matrix with entries $a_{ij}$ such that $a_{jj} = p^{n_j}$ where $n_j \ge 0$ and each $a_{ij}$ with $i<j$ is an integer satisfying $0 \le a_{ij} < p^{n_j}$.
\end{prop}
The matrix $H$ in Proposition \ref{Zp_HNF} is the $p$-adic version of an integer matrix in Hermite normal form.  We call $H$ the \emph{Hermite normal form of $M$} and write $\HNF(M) = H$.  We say that $H \in M_d(\Z_p)$ of this type is in {\em Hermite normal form}.  

The left multiplication by the matrix $U$ corresponds to operations on the rows of $M$ that are operations on the standard basis vectors $e_i$.  The invertibility of $U$ over $\Z_p$ ensures that the rows of $M$ and the rows of $H$ generate the same sublattice of $\Z_p^n$.  Therefore, if $M = UH$ are as in Proposition \ref{Zp_HNF}, then $\coker(M) \cong \coker(H)$.

\begin{prop}\label{count_Zp_HNF}
Suppose $H \in M_d(\Z_p)$ is in Hermite normal form and $\det(H) = p^e$. Then,
\begin{eqnarray*}\label{eq_HNF}
\Prob_{M \in M_d(\Z_p)}(\HNF(M) = H) & = & \frac{\Prob_{M\in M_d(\Z_p)}\left(v_p(\det(M))  =   v_p(\det(H))\right)}{\#\left\{\Lambda \subseteq \Z^d \colon [\Z^d \colon \Lambda] = \det(H)\right\}} \\
& = &  \frac{\prod_{j=1}^d(1-p^{-j})}{p^{ed}}.
\end{eqnarray*}
\end{prop}

\begin{proof}
Since $M = U \HNF(M)$ for some $U \in \GL_d(\Z_p)$, and $\det(U)$ is a unit in $\Z_p$, we see that $v_p(\det(M)) = v_p(\det(\HNF(M)))$.  The set of $M \in M_d(\Z_p)$ with $v_p(\det(M)) = p^e$ is a disjoint union of orbits $\GL_d(\Z_p) H_1,\ldots, \GL_d(\Z_p)  H_k$, where $H_1,\ldots, H_k$ are the finitely many distinct matrices in Hermite normal form of determinant $p^e$.  These matrices are in bijection with the sublattices of $\Z_p^d$ of index~$p^e$, which are in bijection with the sublattices of $\Z^d$ of index~$p^e$.

The volume of the orbit $\GL_d(\Z_p) H_i$ is equal to the probability that $d$ randomly chosen vectors of $\Z_p^d$ each lie in the lattice spanned by the rows of $H_i$ and are linearly independent.  This probability depends on $\det(H_i)$, but not on the particular choice of $H_i$.  This completes the proof of the first equality in Proposition \ref{count_Zp_HNF}.

Proposition \ref{CL_cor} gives $\Prob_{M \in M_d(\Z_p)}\left(v_p(\det(M))= v_p(\det(H))\right)$.  We know that 
\[
\#\left\{\Lambda \subseteq \Z^d \colon [\Z^d \colon \Lambda] = \det(H)\right\}
\] 
is the $p^{-es}$ coefficient of the power series expansion of $\zeta_{\Z^d}(s) = \zeta(s)\zeta(s-1)\cdots \zeta(s-(d-1))$, which is $\pbinomreg{d+e-1}{e}$.  

For the final equality in Proposition \ref{count_Zp_HNF}, recall that
\[
\pbinomreg{d+e-1}{e} = \pbinom{d+e-1}{e} p^{e(d-1)},
\] 
and therefore,
\[
\frac{\frac{\prod_{j=1}^d (1-p^{-j})}{p^e} \pbinom{d+e-1}{e}}{\pbinomreg{d+e-1}{e}} = \frac{\prod_{j=1}^d (1-p^{-j})}{p^{ed}}.
\]
\end{proof}

By Proposition \ref{hnf_prop}, a sublattice $\Lambda \subseteq \Z^d$ gives rise to a $d\times d$ integer matrix in Hermite normal form $H(\Lambda)$ with $[\Z^d\colon \Lambda] = \det(H(\Lambda))$.  An application of the Chinese remainder theorem shows that for each prime $p,\ H(\Lambda)$ determines a matrix in Hermite normal form with determinant equal to a power of $p$. 
\begin{definition}\label{HNFp_def}
Suppose $H \in \mathcal{H}_d(\Z)$ has entries $a_{ij}$ and $\det(H) = p_1^{e_1}\cdots p_r^{e_r}$ where $p_1, \ldots, p_r$ are distinct primes and each $e_i$ is a positive integer.  For each prime $p_i$, define $H_{p_i} \in \mathcal{H}_d(\Z)$ with $\det(H_{p_i}) = p_i^{e_i}$ as follows.  Write each $a_{jj} = p_i^{b_j} u_j$ where $p_i \nmid u_j$.  Let $H_{p_i}$ be the upper-triangular matrix with entries $b_{jk}$ defined so that:
\begin{itemize}
\item $b_{jj} = p_i^{b_j}$, and
\item for $1\le j < k \le d,\ b_{jk}$ is the unique integer $0 \le b_{jk} < p_i^{b_k}$ with $b_{jk} \equiv \frac{a_{jk}}{u_j} \pmod{p_i^{b_k}}$.
\end{itemize}
If $p$ is a prime for which $p\nmid \det(H)$, define $H_p = I_d$.
\end{definition}
\noindent With this definition, it is clear that if $H \in \mathcal{H}_d(\Z)$, then each $H_p \in \mathcal{H}_d(\Z)$ as well.  The following proposition follows from an application of the Chinese remainder theorem.
\begin{prop}
Definition \ref{HNFp_def} gives a bijection between matrices $H \in \mathcal{H}_d(\Z)$ and collections $\left\{H_p\right\}$ consisting of one $H_p \in \mathcal{H}_d(\Z)$ for each prime $p$ and where all but finitely many $H_p = I_d$.  Moreover, $\det(H) = \prod_p \det(H_p)$.
\end{prop}

Note that for any $H \in \HH_d(\Z),\ \coker(H)_p = \coker(H_p)$.  That is, the Sylow $p$-subgroup of the cokernel of $H$ depends only on $H_p$.  Suppose $Q \in \HH_d(\Z)$ has $\det(Q) = p^e$.  Theorem \ref{dist_ppart} follows from showing that the probability that a random $H \in \mathcal{H}_d(\Z)$ has $H_p = Q$ is equal to the probability that a random $M \in M_d(\Z_p)$ has $\HNF(M) = Q$.  
\begin{prop}\label{count_HNF_p}
Suppose $Q \in \mathcal{H}_d(\Z)$ satisfies $\det(Q) = p^e$ for some prime $p$ and positive integer $e$.  Then
\[
\lim_{X \to \infty} \frac{\#\left\{H \in \HH_d(X) \colon H_p = Q\right\}}{\#\HH_d(X)} = \frac{\prod_{j=1}^d (1-p^{-j})}{p^{ed}}.
\]
\end{prop}

\begin{proof}
There is a bijection between $\left\{H \in \mathcal{H}_d(X) \colon H_p = Q\right\}$ and $\left\{H \in \mathcal{H}_d\big(\frac{X}{p^e}\big) \colon H_p = I_d\right\}$.  Note that 
\begin{eqnarray*}
\lim_{X \to \infty} \frac{\#\left\{H \in \mathcal{H}_d\big(\frac{X}{p^e}\big) \colon H_p = I_d\right\}}{\#\HH_d(X)}  & = & 
\lim_{X \to \infty} \frac{\#\left\{H \in \mathcal{H}_d\big(\frac{X}{p^e}\big) \colon H_p = I_d\right\}}{\#\HH_d\left(\frac{X}{p^e}\right)}  \frac{\#\HH_d\left(\frac{X}{p^e}\right)} {\#\HH_d(X)}.
\end{eqnarray*}
By \eqref{eq:Zd_asymptotic}, we have that 
\[
\lim_{X \to \infty} \frac{\#\HH_d\big(\frac{X}{p^e}\big)} {\#\HH_d(X)} = \frac{1}{(p^e)^d}.
\]
Let $g_d(k)$ be the number of sublattices $\Lambda \subseteq \Z^d$ of index $k$ for which $H(\Lambda)_p = I_d$.  We have 
\begin{equation}\label{ct1}
\sum_{k =1}^\infty g_d(k) k^{-s} = \left((1-p^{-s})(1-p^{-s+1})\cdots (1-p^{-s+d-1}) \right)\zeta_{\Z^d}(s).
\end{equation}
This expression is the same $\zeta_{\Z^d}(s)$ except that the local factor at $p$ in its Euler product is replaced with $1$.  It is still clear that the right-most pole of this function is a simple pole at $s=d$, so applying a Tauberian theorem as we did to derive expression \eqref{eq:Zd_asymptotic} shows that 
\begin{equation}\label{ct2}
\lim_{X \to \infty} \frac{\#\left\{H \in \mathcal{H}_d(\frac{X}{p^e}) \colon H_p = I_d\right\}}{\#\HH_d\left(\frac{X}{p^e}\right)} = (1-p^{-d})(1-p^{-d+1})\cdots (1-p^{-1}).
\end{equation}
Combining equations \eqref{ct1} and \eqref{ct2} completes the proof.

\end{proof}
Combining Propositions \ref{FW_prop}, \ref{count_Zp_HNF}, and \ref{count_HNF_p} completes the proof of Theorem \ref{dist_ppart}.

\section{Conclusion}
\label{sec:conclusion}

The results and methods of this paper suggest several natural
directions for further study.

\subsection{Subgroup and subring growth zeta functions}
We may also try to construct multivariate
Dirichlet series to study subgroup growth for other groups.  A first
case of potential interest is the discrete Heisenberg group
\[H_3=\langle a,b,c|\ c=[a,b], [a,c]=[b,c]=1\rangle.
\]
The {\em normal subgroup zeta function} of $H_3$ is
\begin{equation}
  \label{eq:h3}
  \zeta_{H_3}(s)=\sum_{N\trianglelefteq_f H_3} [H_3:N]^{-s}
\end{equation}
where the sum is over all finite index normal subgroups of  $H_3$.
It has been shown \cite{gss}  that
\begin{equation}
  \label{eq:h3b}
  \zeta_{H_3}(s)= \zeta(s)\zeta(s-1)\zeta(3s-2).
\end{equation}
A multivariate generalization of this series might give more refined
information on the distribution of the finite groups which arise as
quotients of $H_3$.

Similar questions can be asked for subring growth.  For example, we
expect that the cotype subring zeta function of $\Z^3$ can be used
to show that in contrast to the  case studied here, most of the
subrings of $\Z^3$ (ordered by index) are not cocyclic. In a nonabelian setting, the Lie ring
$\mathfrak{sl}_2(\Z)$ has an explicitly computed zeta function
\begin{equation}
  \label{eq:sl2Z}
  \zeta_{\mathfrak{sl}_2(\Z)}(s)=
\sum_{L}[\mathfrak{sl}_2(\Z):L]^{-s}=
P(2^{-s})\frac{\zeta(s)\zeta(s-1)\zeta(2s-2)\zeta(2s-1)}{\zeta(3s-1)},
\end{equation}
where the sum is over all finite index Lie subrings of
$\mathfrak{sl}_2(\Z)$ and $P(x)=(1+6x^2-8x^3)/(1-x^3)$ \cite{duST}.
It would be interesting to compute the cotype subring zeta function of
$\mathfrak{sl}_2(\Z)$ and use it to find the density
of Lie subrings with cyclic quotient.

Klopsch and Voll compute the subring zeta functions of all 3-dimensional Lie
algebras over $\Z_p$ in a
uniform manner \cite{KlopschVoll3d}.  Their techniques, in particular, should allow
one to compute the cotype zeta function for both $H_3$ and
$\mathfrak{sl}_2(\Z)$. 

\subsection{Zeta functions of classical groups}
\label{subsec:classical}

The subgroup growth zeta function $\zeta_{\Z^d}(s)$ of $\Z^d$
also arises in the more general context of the zeta functions associated
to algebraic groups studied by Hey, Weil, Tamagawa, Satake, Macdonald
and Igusa \cite{He, W, T, S, M, I}.  For $G$ a linear algebraic
group over $\Q_p$ and a rational representation $\rho:G\to \GL_n
$ they define
\begin{equation}
  \label{eq:zfcg}
  Z_{G,\rho}(s)=\int_{G^+}|\det\rho(g)|^s\,dg
\end{equation}
where $G^+=\rho^{-1}(\rho(G(k)\cap M_n(\oo_p)),$ where
$\oo_p$ is the ring of integers of $\Q_p.$ When $G=\GL_n$ and $\rho$
is the natural representation, $Z_{G,\rho}(s)$ is just the
$p$-part of the subgroup growth zeta function $\zeta_{\Z^d}(s).$ In
more recent work, du Sautoy and Lubotzky \cite{duSL} show that
$Z_{G,\rho}(s)$ for more general $G$ and $\rho$ continues to have an
interpretation as a generating series counting substructures of
algebras.


We take an explicit example from Bhowmik-Grunewald \cite{bg}, see
also \cite[Theorem 12]{bhowmik}.  
Let $\beta$ be the alternating bilinear form on a $2n$
dimensional space associated to the matrix
\[
\begin{pmatrix}
  0& -I_n\\I_n&0
\end{pmatrix}.
\]
A sublattice $\Lambda$ of $\Z^{2n}$ is $\beta$-polarized if
$\widehat\Lambda=c\Lambda$ for some constant $c\in \Q^\times$, where 
\[
\widehat\Lambda=\{v\in \Z^{2n}: \beta(u,v)\in \Z \mbox{ for all } 
u\in\Z^{2n}\}.
\]
Define the group $\GSp_{2n}(\Q)$ of symplectic similitudes by 
\begin{equation*}
  \label{eq:gsp}
  \GSp_{2n}(\Q)=\{g\in \GL_{2n}(\Q): \beta(gx,gy)=\mu_g\beta(x,y) 
\mbox{ for some $\mu_g\in\Q^\times$ and all $x,y\in \Q^n$}\}.
\end{equation*}
Following computations of Satake \cite{S} and Macdonald \cite{M}, the
zeta function of the group $ \GSp_{6}(\Q)$ is written down explicitly
in \cite{duSL}.  Bhowmik and Grunewald use this to show that the
number of $\beta$-polarized sublattices of $\Z^6$ of index less than
$X$ is asymptotic to $cX^{7/3}$ for an explicit constant $c$.  The
results of \cite{duSL} indicate a way to extend these computations, both to higher
rank and to include the distribution of cotype.

\bibliographystyle{hplain}
\bibliography{sources}

\end{document}